\documentclass[11pt,leqno]{amsart}

\usepackage{amssymb,amsfonts,amsmath}
\usepackage[T1]{fontenc}

\newtheorem{theorem}{Theorem}
\newcounter{lemm}
\newtheorem{lemma}[lemm]{Lemma}
\newcounter{prop}
\newtheorem{proposition}[prop]{Proposition}
\newcounter{coro}
\newtheorem{corollary}[coro]{Corollary}

\def\dim{\mathop{\rm dim}\nolimits}
\def\codim{\mathop{\rm codim}\nolimits}
\def\deg{\mathop{\rm deg}\nolimits}

\newcounter{rema}
\newenvironment{remark}{\smallskip\noindent\stepcounter{rema}{\bf
Remark        \arabic{rema}}.}{\qed\smallskip} 
\newcounter{exam}
\newenvironment{example}{\smallskip\noindent\stepcounter{exam}{\bf
Example \arabic{exam}}.}{\qed\smallskip}

\begin{document}

\begin{title}
{Note on classical notion of   Lee form.}
\end{title}

\begin{author}
{Piotr Dacko} 
\end{author}

\email{{\tt piotrdacko{\char64}yahoo.com}}

\begin{abstract}
This note is devoted to partial study of recurrent equation 
$d\omega=\beta\wedge\omega$, based on  linear algebra of
 exterior forms. Such equation was considered 
by  Lee, for non-degenerate 2-form. In this note we approach 
general case, when $\omega$ is arbitrary.  Particularly, we 
extend  results obtained by Lee, on odd-forms. 
\end{abstract}

\maketitle

\section{Introduction} It was noticed by Lee, \cite{Lee}, that
equation $d\omega=\beta\wedge\omega$, for given non-degenerate
exterior 2-form follows:  a) $d\beta=0$, if dimension of manifold is
$\geq 6$, b) if dimension $=4$, then  for any 3-form $\kappa$, there
is exactly 1-form $\beta$, such that  $\kappa= \beta\wedge\omega$. In
the original paper, there is very simple justification of the latter
fact: $\kappa=\beta\wedge\omega$, is equivalent to system of linear
equations, which can be  resolved uniquely. Particularly, in dimension
four we can find $\omega$, with $d\beta\neq 0$.  Later on Libermann,
also \'Slebodzi\'nski rediscovered these results, \cite{Lib1,Sleb}.

In this note we try to obtain some information in general setting,
 when $\omega$ is arbitrary.  The starting observation is that 
the equation $d\omega=\beta\wedge\omega$, implies $d\beta\wedge\omega=0$. 
The latter can be studied point-wise, by means of some basic linear
 algebra.  
\section{Preliminaries}

For vector space $V$, $\dim V=n$, by $\Lambda^k(V^*)$ we denote the
space of all $k$-linear totally anti-symmetric real functions (forms) on
$V$. We set $\Lambda^0(V^*)=\mathbb R$ and $c\in \mathbb R$ is treated as
constant function. Elements of
$\Lambda^k(V^*)$ are called $k$-forms, we set 
\begin{equation}
\Lambda(V^*) =  \Lambda^0(V^*)\oplus\Lambda^1(V^*)\oplus\ldots,
\oplus\Lambda^n(V^*).
\end{equation} The $\wedge$-product (exterior multiplication) on
$\Lambda(V^*)$ is defined as usually. If $c\in \Lambda^0(V^*)$,
$\omega\in\Lambda(V^*)$ then $c\wedge\omega = c\,\omega $.  The degree
$\deg \beta$, of a form is defined as a number of its arguments, by
definition $\deg c =0$, $c\in \mathbb R$. Interior multiplication of a
$(k+1)$-form $\theta$ and a vector $v\in V$,   is a $k$-form  $
\iota_v\theta$, defined by the equation
\begin{equation}
(\iota_v\theta)(u_1,\ldots,u_k) = (\deg \theta)\,\theta(v,u_1,\ldots,u_k)
\end{equation} where $u_1,\ldots,u_k \in V$.

For arbitrary forms $\iota_x(\mu\wedge\nu) = \iota_x\mu\wedge\nu
+(-1)^{deg \mu}\mu\wedge\iota_x\nu$.  For any vector $\iota_x^2 =0$, however
$\iota_x$ is exact in the sense that if $\iota_x\mu=0$, then there is a
form $\nu$, $\deg \nu = \deg \mu +1$, and $\iota_x\nu=\mu$. For vectors
$x_1,\ldots,x_k\in V$, we set $\iota_{[x_1x_2\ldots
x_k]}=\iota_{x_1}\iota_{x_2}\ldots\iota_{x_k}$, changing order results
$\iota_{[x_{i_1}\ldots x_{i_k}]} = \pm \iota_{[x_1\ldots x_k]}$, where
$\pm $ is sign of permutation $(1\mapsto i_1,\ldots ,k\mapsto
i_k)$. The operator $\iota_{[x_1\ldots x_j]}$, $j \geq 1$, we call
j'th-derivative and denote $\iota^{(j)}$.

With help of interior multiplication we can define pairing between
k-tuples $(x_1,\ldots,x_k)$ of vectors  and k-forms, by the formula 
\begin{equation}
<[x_1\ldots x_k],\theta> = 
\iota_{[x_1\ldots x_k]}\theta = \pm k!\,\theta(x_1,\ldots,x_k),
\end{equation}
on the right hand side $\pm$, is  sign of the reverse 
$(1\mapsto k, 2\mapsto (k-2),\ldots, k\mapsto 1)$. 
From this definition follows  that pairing is non-degenerate: if
\[
<[x_1\ldots x_k],\theta>=0,
\] for any k-tuple, then $\theta=0$.  For
$\theta=\eta_1\wedge\ldots\wedge\eta_k$  simple
\begin{equation}
<[x_1\ldots x_k],\theta> = \pm\det |\eta_i(x_j)|,\quad i,j=1,\ldots,k.
\end{equation}  If
$(e_1,\ldots,e_n)$, is an ordered base of $V$, and order is extended to dual forms
$(\alpha_1,\ldots,\alpha_n)$, $\alpha_i(e_j)=\delta_{ij}$, then 
\begin{equation*}
\iota_{[e_{i_1}\ldots e_{i_k}]}\theta, \quad i_1 < \ldots < i_k,
\end{equation*}
is  coefficient of $\beta$, at term $\alpha_{i_1}\wedge\ldots \wedge \alpha_{i_k}$.

For  vector subspace $C \subset V$, $0 \leq p=\dim C < n$, let  
\begin{equation}
C^0 = \left\{ \alpha \in V^*\; |\; \alpha(x) =0, \;x\in C\right\},
\end{equation} be a space of all 1-forms  vanishing
on $C$. For $\dim C =0$, $C^0=V^*$. We set $k=\dim C^0$,
 then $k=\codim C = n-p$. 

We are interested in studying properties of forms related to the pair
$(V,C)$. By definition $(V,\lbrace 0\rbrace) = V$. Isomorphism of
$(V,C)$ is  non-degenerate linear map of $V$, leaving $C$
invariant. The linear base $(e_1,\ldots,e_p, \ldots, e_n)$ of $(V,C)$
is a  base of $V$, where $(e_1,\ldots,e_p)$, span $C$. By  change of a
base, it is understood passing from base to base of $(V,C)$.  In
similar manner  we understand isomorphisms and linear bases of
$(V^*,C^0)$.  Usual duality of linear maps $f \leftrightarrow f^*$,
$f: V\rightarrow V$,  $f^*: V^*\rightarrow V^*$, $f^*\alpha = \beta$,
$\beta(x)=\alpha(fx)$, $x\in V$, establishes duality of pairs $(V,C)$
and $(V^*,C^0)$: $f$ is isomorphism  of $(V,C)$ if and only if $f^*$
is isomorphism of $(V^*,C^0)$. The same is true  for the operation of
base change: by duality changing bases of $(V,C)$ is equivalent  to
changing bases of $(V^*,C^0)$.  We will use these facts without
explicitely  referring to them.

The pair
$(V,C)$ gives rise to properly defined sub-algebra
$\Lambda(C^0)\subset \Lambda(V^*)$.  If $(\alpha_1,\ldots,\alpha_k)$
is a base  of $C^0$, then $\Lambda^1(C^0)=C^0$, and $\tau \in
\Lambda^l(C^0)$, $l\geq 2$, means, that $\tau$ is a sum of
$\wedge$-products of $\alpha_i$'s.  Any element
$\tau\in\Lambda^l(C^0)$, $l\geq 1$, nullifies $C$,  in the sense that
$\iota_{[v_1\ldots v_j]}\tau=0$, whenever at least one of
$v_1,\ldots, v_j$ is  in $C$.

Let the $(\beta_1,\dots,\beta_{n-k})$ be such that
$(\alpha_1\ldots,\alpha_k,\beta_1, \ldots, \beta_{n-k})$ is a base of
$V^*$.  A form $\omega$  can be written
as a sum 
\begin{equation}
\label{decomp}
\omega =\mu_{s_1}+\mu_{s_2}+\dots , \quad s_1 < s_2 < \ldots, 
\end{equation}
where
\begin{equation}
\mu_{s_l}=  \sum\limits_{\sigma=(i_1< \ldots <i_{s_l})} 
 \beta^\sigma_{(s_l)}\wedge\alpha_{i_1}\wedge\ldots\wedge
\alpha_{i_{s_l}},\quad l=1,2,\ldots,
\end{equation} 
\begin{equation}
\beta^\sigma_{(s_l)} = \sum\limits_{j_1 < \ldots < j_{m_l}}c_{\sigma,s_l}^{j_1\ldots j_{m_l}}
\beta_{j_1}\wedge\ldots\wedge\beta_{j_{m_l}}, \quad
c_{\sigma,s_l}^{j_1\ldots j_{m_l}}=const.,
\end{equation}
and 
\[
 s_l+m_l = \deg \omega ,\quad l=1,2,\ldots .
\]
 The form $\mu_{s_1}$ will
be called  main part of $\omega$ and we denote it as $\omega^*$, and
by $|\omega^*|$  the common number of $\alpha_i$'s in each summand 
of the main part, so 
\begin{equation}
|\omega^*| =\deg (\alpha_{i_1}\wedge\ldots\wedge\alpha_{i_{s_1}}) = s_1.
\end{equation}
The difference $\omega
-\omega^*$, is a {\em reminder}.  Such decomposition depends on the
choice of the  base of $(V^*,C^0)$, however  if
\[
\omega =\mu_{s'_1}+\mu_{s'_2}+\dots,
\] is the respective decomposition in some other base, then $s_1 = s'_1$.

\begin{proposition} 
\label{auxprop} For any $(j+s)$-form $\omega \neq 0$,  $s=|\omega^*|$,
$1 \leq j \leq  \dim C$, there is derivative $\iota^{(j)}$, such that
$\iota^{(j)}\omega \neq 0 \in \Lambda^s(C^0)$.  
\end{proposition}   
\begin{proof}  We fix  base of $(V^*,C^0)$,  let  $1 \leq j \leq \dim
C$ be  the common degree $j=\deg \beta^{\sigma}$, of coefficients of
main part $\omega^*=\sum_\sigma \beta^\sigma\wedge\alpha_\sigma$.
For vectors $v_1,\ldots,v_j \in C$  
\begin{equation}
\omega'=\iota_{[v_1\ldots v_j]}\omega^*= 
\sum\limits_{\sigma}<[v_1\ldots v_j],\beta^\sigma>\alpha_\sigma,
\end{equation} let $ c^{\sigma}= <[v_1\ldots
v_j],\beta^{\sigma}>$, then for some $v_1,\ldots,v_j$, at least
one $c_\sigma$ is non-zero. As forms $\alpha_{\sigma}\in
\Lambda^s(C^0)$, in the decomposition of $\omega^*$, are linearly
independent, we have $\omega'\neq 0$, $\omega'\in \Lambda^s(C^0)$. Now
it is enough to notice, that all $j$-derivatives $\iota_{[v_1\ldots
v_j]}$, $v_1,\dots, v_j \in C$, of the reminder of $\omega$ are zero.
\end{proof}   
\begin{remark} In the case $|\omega^*|=0$ this result states that
$\iota^{(l)}\omega=const \neq 0$.
 \end{remark}
% \begin{corollary} Let 
% \begin{equation}
% N(\omega) = \lbrace v \in V\;|\; \iota_v\omega =0 \rbrace, 
% \end{equation}
% then for  $\omega_1, \omega_2\in \Lambda(V^*)$  
% \begin{equation}

% \omega_1\wedge\omega_2 =0 \; \Rightarrow \; 
% \lbrace v\in V\;|\;\iota_v\omega_1= \iota_v\omega_2 =0 \rbrace \neq 
% \lbrace 0 \rbrace,
% \end{equation} the space of vectors which nullify both $\omega_1$
% and $\omega_2$ is non-trivial.
% \end{corollary} 

\section{\'Slebodzi\'nski Lemma} Let $\Omega \neq 0$ be a 2-form on
$V$, $ \dim V=n+2p \geq 3$, where $p\geq 1$, is a rank of $\Omega$, so $p$ is maximal
integer such that $\Omega^{\wedge p}\neq 0$. The space $C$ is now
the kernel of $\Omega$
\[
C = \left\{ x \in V\;| \; \iota_x\Omega =0 \right\}, 
\]
then $\dim C =n$, $\dim C^0=2p$.
 Clearly $\Omega \in \Lambda(C^0)$. For $C =\lbrace 0 \rbrace $, we have
$\Lambda(C^0)=\Lambda(V^*)$. 
We want to answer  the question, under what  conditions the
equation 
\begin{equation}
\label{omwbet}
\Omega\wedge\beta =0,
\end{equation}
 has non-trivial solution $\beta\neq 0$. 

\begin{lemma}
\label{alglem}
Let $1 \leq l \leq \dim V -2$ and $ 0 \leq s\leq \min (2p,l)$, define
\begin{equation}
\mathcal K_{l,s} =\left\{ \beta \in \Lambda^l(V^*)\setminus \lbrace 0 \rbrace \;|\; 
\Omega\wedge\beta = 0,\, s=|\beta^*| \right\}, 
\end{equation}
then
\begin{gather}
\mathcal K_{l,s} =\emptyset, \quad 0\leq s < \min (p,l),  \\
\mathcal  K_{l,s} \neq \emptyset, \quad s \geq p . 
\end{gather}
\end{lemma}
\begin{proof}
Let  $\beta\neq 0$, be a $l$-form, $l\geq 1$, $s=|\beta^*|$, 
 and define  $\beta'$  as follows: 
\begin{equation}
\beta' = \begin{cases}
\beta,  &\textrm{if $\beta \in \Lambda^l(C^0)$,} \\
\iota^{(j)}\beta\in \Lambda^s(C^0),  &\textrm{ $\iota^{(j)}\beta$ is as in the Proposition {\bf \ref{auxprop}},}
\end{cases}
\end{equation}
here $\iota^{(j)}=\iota_{[v_1\ldots v_j]}$, for some vectors $v_1,\ldots, v_j\in C$.  For $\beta\neq \beta'$ 
\begin{equation}
\label{const}
0= \iota_{[v_1\ldots v_j]}(\Omega\wedge\beta) = 
\Omega\wedge (\iota_{[v_1\ldots v_j]}\beta) = \Omega\wedge \beta'.
\end{equation}  
Thus, in any case we have
\begin{equation}
\Omega\wedge\beta=0 \quad \Rightarrow \quad \Omega\wedge\beta'=0.
\end{equation} The right hand side of this implication follows, that
$\beta'$ can not be a constant, thus $s > 0$.  By induction 
\begin{equation}
\label{indform}
\Omega\wedge\beta'=0 \quad \Rightarrow \quad
 \Omega^{\wedge(i+1)}\wedge\iota_{[v_1\ldots v_i]}\beta'=0,  
\end{equation}
 for any vectors $v_1,\ldots, v_i \in V$. For $s < p$
\begin{equation}
\Omega\wedge\beta'=0 \quad \Rightarrow \quad 
<[v_1\ldots v_s],\beta'> \Omega^{s+1}=0, 
\end{equation} the last equation  follows $<[v_1\ldots
v_s],\beta'>=0$.  Non-degeneracy of $<\cdot,\cdot>$, implies
$\beta'=0$.  Hence $s \geq p$.

 Simple dimension considerations, $ \dim \Lambda^s(C^0) =
\binom{2p}{s} $, follow, that for $s \geq p$ there is $\beta' \neq
0\in \Lambda^s(C^0)$, such that $\Omega\wedge\beta' =0$. Now we may
take $\beta=\tau\wedge\beta'\neq 0 \in \mathcal K_{l,s}$, for some
$(l-k)$-form $\tau$, then $\Omega\wedge\beta =0$.

\end{proof}

\begin{corollary}
\label{degbound} The immediate consequence of the above result is that
degree of $ \beta\neq 0 $, in (\ref{omwbet}) is always bounded below
by the  rank of $\omega$, $\deg \beta \geq p$.   
\end{corollary}

\begin{remark}
\label{rankbound} Let $\omega_1\wedge\omega_2=0$, for $2$-forms
$\omega_1$, $\omega_2$, on a vector space $V^n$. If $\omega_1 \neq 0$
and $\omega_2\neq 0$, then $\omega_1$, $\omega_2$ are  at most of rank
2; there are 1-forms
$(\eta_1,\eta_2,\beta_1,\beta_2)$, such that 
\begin{equation}
\omega_1=\eta_1\wedge\eta_2+c\,\beta_1\wedge\beta_2, \quad c\in \mathbb R,
\end{equation}
and we have the following five possibilities for $\omega_2$
\begin{equation}
\eta_1\wedge\eta_2-c\beta_1\wedge\beta_2,\quad
b\,\eta_i\wedge\beta_j,\quad b\neq 0 \in \mathbb R, \quad i,j=1,2,  
\end{equation} 
\end{remark}

The operation $\beta \mapsto \omega\wedge \beta$ defines family of maps
$\lambda^k :\Lambda^k(V)\rightarrow \Lambda^{k+2}(V)$,
\begin{equation}
\Lambda^k(V)\ni \beta \mapsto \omega\wedge\beta \in \Lambda^{k+2}(V), \quad
k =0,\ldots, n,
\end{equation} 
clearly $\lambda^{n-1}=\lambda^n=0$.
\begin{corollary} Let $\omega$ has trivial kernel, so $\dim V=2p$,
$\omega^{\wedge p}\neq 0$.  Then maps $\lambda^k$ are 1-1, for $k \leq
p-1$. In particular $\lambda^{p-1}$ is isomorphism of the spaces
$\Lambda^{p-1}(V)$ and $\Lambda^{p+1}(V)$.
\end{corollary} For $p=2$, this Corollary rediscovers the result of
Lee, \cite{Lee} (cf. \'Slebodzi\'nski, \cite{Sleb}, p. 314) 

\begin{remark}
It is unknown to the author, but it is expected that for $k\geq p$,  
$\lambda^k$ is epimorphism. Moreover, it is possible in combinatorial 
way, to describe forms spanning the kernel of $\lambda^k$.
\end{remark}

\section{Main result and its applications}
At the beginning let consider two examples.

\begin{example}  Let $d\beta=\eta\wedge\beta$, for non-zero 1-forms
$\beta$ and $\eta$. Notice, that $\eta$ in this equation is
non-unique: two  solutions differs by  $f\beta$,  for a function $f$.
However, $\beta\wedge d\beta=0$, wich means, that the kernel of
$\beta$ is involutive, and locally $\beta=f\beta_0$, for some
closed 1-form $\beta_0$. Then $d\beta= d\ln |f|\wedge\beta $, so we may
set $\eta=d\ln |f|$. In conclusion we obtain, that the system
\begin{equation}
d\beta=\eta\wedge\beta, \quad d\eta=0,
\end{equation}
 determines $\eta$ uniquely.
\end{example}

\begin{example} Let $\omega$ be a 2-form of maximal rank on a manifold
$\mathcal M$.  In other words the kernel of $\omega$ is trivial or
one-dimensional.  Let again, $d\omega = \eta\wedge\omega $.   Here
$\eta$ is unique if $\dim \mathcal M\geq 4$. No additional assumptions
are needed.
\end{example}

On a manifold the rank of a  $2$-form may vary from point to point.
For a $2$-form $\omega$ we denote by $r(\omega)$ function which
associates to each point  the rank of $\omega$ at this point.  
Let $r(\omega)=k$, at a point. Then $\omega^k\neq 0$ and $r(\omega)\geq k$ 
on some neighborhood $\mathcal U$ of this point. This argument 
proves, that $r(\omega)$ is lower semi-continuous on any manifold.

\begin{theorem} Let $\omega$ be  smooth $2$-form on
connected, smooth manifold $\mathcal M$, $\dim \mathcal M \geq
4$, such that
\begin{equation}
d\omega=\beta\wedge\omega,
\end{equation}
for some $1$-form $\beta$. 
It is assumed that the set of points, where $\omega =0$
 is nowhere dense in $\mathcal M$.  Let define
\begin{gather}
\mathcal A = \{  r(\omega) > 2 \},\quad
\mathcal B = \{  d\beta \neq 0,\, \omega \neq 0\}, \quad
\mathcal C =\{ r(\omega) \leq 1\},
\end{gather}
then   
\begin{itemize}
\item[a)] $\mathcal A$ is open;  if non-empty, then $d\beta = 0$ on
$\mathcal A$;
\item[b)] If $\mathcal B$ is non-empty, 
then $1 \leq r(d\beta),\, r(\omega) \leq 2$ on $\mathcal B$;
\item[c)] For $\mathcal C$ nowhere dense $\beta$ is unique.
\end{itemize} 
Particularly, it follows from {\rm a)} and {\rm b)}, 
that $\mathcal A \cap \mathcal B = \emptyset$.
\end{theorem}
\begin{proof} At first we notice that $d\omega=\beta\wedge\omega$
follows $d\beta\wedge\omega =0$, at any point of $\mathcal M$.

\noindent $\Rightarrow$ a) $\mathcal A$ is open comes from
the fact, that  $r(\omega)$ is lower semi-continuous. Let $\mathcal
A\neq \emptyset$. Take $p\in \mathcal A$. Then from $(d\beta)_p\wedge\omega_p
=0$ and Lemma $\ref{alglem}$, cf. also Corollary $\ref{degbound}$,
follow $(d\beta)_p=0$. 

\noindent $\Rightarrow$ b) In the case $\mathcal B\neq \emptyset$, let
$p\in \mathcal B$, for $(d\beta)_p\neq 0$ and
$(d\beta)_p\wedge\omega_p=0$,  by the Remark
$\ref{rankbound}$.1, the ranks satisfy $1 \leq r(d\beta),\,
r(\omega)\leq 2$ at $p$. 

\noindent $\Rightarrow$ c) By assumption set of points where  $r(\omega) \geq 2$,
is open an dense; at each point of this set  $\beta$ is  unique, hence is unique 
everywhere.    
\end{proof}

For $\omega$ of maximal rank  and $\dim \mathcal M \geq 6$,
we have $\mathcal A = \mathcal M$.

\begin{corollary} 
\label{corgen}
Let $\dim \mathcal M \geq 6$, and $\omega$ be
a  $2$-form of maximal rank, such that $d\omega=\beta\wedge\omega$. Then $\beta$ is closed,
$d\beta=0$.
\end{corollary}  

In particular case of even-dimensional manifolds we
can restate the following result
\begin{corollary}[Lee-Libermann-\'Slebodzi\'nski, \cite{Lee,Lib1,Sleb}]
Let $\omega$ be a non-degenerate $2$-form on even-dimensional manifold
$\mathcal M$. Assume that $d\omega=\beta\wedge\omega$. If $\dim
\mathcal M \geq 6$, then $\beta$ is closed, $d\beta=0$. 
\end{corollary}

We emphasize that the Corollary {\bf \ref{corgen}}, is an enhancement
of the above mentioned result, for it holds also for odd-dimensional
manifolds.  
 
\begin{remark} 
 In general, let $(\mathcal M, J,g)$, be an almost Hermitian
manifold, where $\omega$ is now the fundamental form of $\mathcal M$,
$\omega(X,Y)=g(JX,Y)$, and
\begin{equation}
\label{lcK}
d\omega=\beta\wedge\omega, \quad d\beta=0,
\end{equation} so, $\beta$ is closed. If we focus  only on a
sufficiently small open disk $\mathcal D \subset\mathcal M$, then
$\beta_{\mathcal D}=\beta|_{\mathcal D}$ is exact on $\mathcal D$,
$f:\mathcal D \rightarrow \mathbb R$, $df = \beta_{\mathcal D}$.  Now,
the structure $(J|_{\mathcal D}, e^{-f}g|_{\mathcal D})$ is an almost
K\"ahler structure on $\mathcal D$.  According to Vaisman \cite{Vais1},  such
manifolds are called locally conformal (almost)  K\"ahler,
(l.c.a.K. manifolds). The form $\beta$ is called Lee
form.   

In modern literature, many authors, when referring the notion of
l.c.a.K manifolds, are using (\ref{lcK}). Of course such definition is
correct but redundant and, worse, can be confusing, suggesting that 
$d\beta=0$ is general requirement. If $\dim \mathcal M \geq 6$,  once the
fundamental form  satisfies $d\omega=\beta\wedge\omega$, the Lee form
is automatically closed, $d\beta=0$.
\end{remark}

\begin{remark} An almost contact metric manifold $(\mathcal
M,\phi,\xi,\eta,g)$ is called almost $\alpha$-cosymplectic \cite{HAKAN}, if 
\begin{equation} 
\label{alalcos}
d\eta=0,\quad d\Phi = 2\alpha\eta\wedge\Phi,
\end{equation}  where $\Phi$ is fundamental form of $\mathcal M$,
$\Phi(X,Y)=g(\phi X,Y)$ and  $\alpha$ is a function on $\mathcal
M$. Form $\Phi$ is non-degenerate and the Reeb vector field $\xi$,
spans kernel of $\Phi$, $\iota_\xi\Phi=0$.
 Particularly, for $\alpha =1 $,
(\ref{alalcos}) defines class of almost Kenmotsu
manifolds, \cite{DilPast}. 
 
If $\dim \mathcal
M > 5$, then the condition $d\Phi=2\alpha\eta\wedge\Phi$, yields, that
 $\alpha\eta$, is closed
\begin{equation} 
d(\alpha\eta)= d\alpha\wedge\eta =0.
\end{equation} Hence $d\alpha = f\eta$, $f = \xi\alpha$.  There are two simple 
remarks:
\begin{itemize}
\item[a)] there is no need to require $d\eta=0$, in the definition of 
almost Kenmotsu manifolds, for dimensions $> 5$,
\item[b)] 
let drop the assumption $d\eta=0$, in
(\ref{alalcos}), nevertheless, near points where $d\alpha\neq 0$, we have
$d\eta\wedge\eta=0$, so the kernel distribution $\eta = 0$, on such
domains is completely integrable.
\end{itemize}   
\end{remark}

\begin{example} Let 
\begin{equation}
\omega_f=e^fdx^1\wedge dx^2 + dy^1\wedge dy^2,
\end{equation}
be defined on $\mathbb R^4$, $v=(x^1,x^2,y^1,y^2)\in \mathbb R^4$.  
Then
$d\omega_f=\beta_f\wedge\omega_f$, and for particular choices of the
function $f$, the form $\beta$ satisfies a priori imposed conditions.  
Set  $f=f_0=x^1y^1+x^2y^2$, $\omega_0=\omega_{f_0}$, then  
\begin{equation}
\beta_0 = x^1dy^1+x^2dy^2,\quad d\beta_0=dx^1\wedge dy^1+dx^2\wedge dy^2,
\end{equation}
clearly $d\beta_0\wedge \omega_0 =0$.  
\end{example}

\begin{example} On the basis of the previous example, let construct on
a manifold  $\mathcal M = \mathbb R_+\times \mathbb R^4$, $p=(t,v)\in
\mathcal M$, $t >0$, $v\in \mathbb R^4$,  a structure consisting of
a pair $(\eta,\Phi)$, where $\eta$ is a $1$-form, $d\eta=0$, $\Phi$ is
a $2$-form, and $\eta\wedge\Phi^{\wedge 2}$ is a volume (oriented) on
$\mathcal M$. Moreover
\begin{equation}
d\Phi = \gamma\wedge\Phi, \quad 
\textrm{and $\gamma$ is a contact form on $\mathcal M$}.
\end{equation} Directly, we verify that forms $\eta=dt$,
$\Phi=t\,\omega_o$, satisfy the required conditions,  
\begin{equation}
d\Phi= \gamma\wedge \Phi,  \quad 
\gamma = d\ln t+\beta_0= d\ln t +x^1dy^1+x^2dy^2, 
\end{equation}
now it is clear, that $\gamma$ is contact form. 
\end{example}

In \cite{Ol}, there is given example of 
four-dimensional Lie group with $\omega$ left-invariant, non-degenerate, 
and $d\omega = \beta\wedge\omega$, $d\beta\neq 0$.

\end{document}